\documentclass[12pt]{amsart}

\usepackage{amsmath}
\usepackage{amscd}
\usepackage{eucal}
\usepackage{amssymb}
\usepackage{epic}
\usepackage{graphics}
\usepackage{color}
\usepackage{mathrsfs}
\usepackage{multirow}
\usepackage{subfigure}
\usepackage{graphicx}
\usepackage{mathtools}

\numberwithin{equation}{section}
\numberwithin{table}{section}

\theoremstyle{plain}
\newtheorem{theorem}{Theorem}[section]

\theoremstyle{remark}

\newtheorem*{remarks}{Remarks}

\newcommand{\CC}{{\mathcal{C}}}

\newcommand{\Zp}{{\mathbb{Z}_{p}}}
\newcommand{\Qp}{{\mathbb{Q}_{p}}}

\newcommand{\ZZZ}{{\mathbb{Z}}}

\DeclareMathOperator{\im}{Im}
\DeclareMathOperator{\irred}{Irred}
\DeclareMathOperator{\diag}{diag}

\title[Abelian Cayley Graphs]{Integer Invariants of Abelian Cayley Graphs}
\author[Ducey]{Joshua E. Ducey} 
\author[Jalil]{Deelan M. Jalil}

\address{Dept.\ of Mathematics and Statistics, James Madison University, Harrisonburg, VA 22807}
\email{duceyje@jmu.edu}

\address{Dept.\ of Mathematics and Statistics, James Madison University, Harrisonburg, VA 22807}
\email{jalil2dm@jmu.edu}


\begin{document}
\begin{abstract}
Let $G$ be a finite abelian group, let $E$ be a subset of $G$, and form the Cayley (directed) graph of $G$ with connecting set $E$.  We explain how, for various matrices associated to this graph, the spectrum can be used to give information on the Smith normal form.  This technique is applied to several interesting examples, including matrices in the Bose-Mesner algebra of the Hamming association scheme $H(n,q)$.  We also recover results of Bai and Jacobson-Niedermaier-Reiner on the critical group of a Cartesian product of complete graphs.
\end{abstract}
\maketitle
\section{Introduction}
Throughout this paper $G$ will denote a finite abelian group (written multiplicatively) and $E$ will denote a subset of $G$.  We can then define a directed graph $\CC$ with vertex set $G$, and an edge from $h$ to $g$ if and only if $gh^{-1} \in E$.  We will refer to $\CC$ as the Cayley graph of $G$ with respect to the \textit{connecting set} $E$.  Note that $\CC$ will be an undirected graph precisely when $E = E^{-1}$, where $E^{-1} = \{e^{-1} \, | \, e \in E \}$.

The purpose of this paper is to apply results of MacWilliams-Mann \cite{M-M} and Sin \cite{sin} in order to obtain the Smith normal form of various matrices attached to many examples of these Cayley graphs.  The structure of the paper is as follows. In the second section we explain basic terminology related to graphs and the Smith normal form of an integer matrix.  In the third section we present and prove the results that motivate our computations.  In the final section we apply these results to numerous examples.
\section{Preliminaries}
Generally speaking, when studying a graph one technique is to encode the information into a matrix, and then study certain numerical and algebraic properties of this matrix.  Properties that remain the same up to isomorphism of graphs are called \textit{invariants}.  We now recall some popular matrices and invariants.

Ordering the vertices of a graph in some fixed (but arbitrary) manner, we can form the \textit{adjacency matrix}, $A$, of the graph:
\[
A(i,j) = 
  \begin{cases}
  1,  & \mbox{if there is an edge from vertex $i$ to vertex $j$}\\
  0, &  \mbox{otherwise,}
  \end{cases} 
  \]
where the symbol $A(i,j)$ denotes the entry of the matrix $A$ corresponding to row $i$ and column $j$. %added notation clarification

We can also consider the \textit{Laplacian matrix}, $L$, defined by
\[
L = D - A,
\]
where $D$ is the \textit{degree matrix}:
\[
D(i,j) = 
  \begin{cases}
  \mbox{the (out) degree of vertex i,} & \mbox{if $i = j$}\\
  0, &  \mbox{otherwise.}
  \end{cases} 
  \]
  
The most well-known invariant of each of these associated matrices is the \textit{spectrum}; that is, the eigenvalues and their multiplicities \cite{B-H}.  Even more fundamental is the \textit{Smith normal form}, which can be defined more generally for (possibly nonsquare) incidence matrices.  We say that two $m \times n$ integer matrices $M$ and $N$ are \textit{equivalent}, and write
\[
M \sim N,
\]
if there exist integer matrices $P$ and $Q$ with determinants $\pm 1$ so that
\[
PMQ = N.
\]
Such matrices $P$ and $Q$ are called \textit{unimodular}, and the condition on their determinants simply forces their inverses to also be integer matrices.

It is well-known that each integer matrix $M$ is equivalent to a matrix $S$ such that
\begin{enumerate}
\item $S(i,j) = 0$ if $i \neq j$, and
\item $S(i,i)$ divides $S(i+1, i+1)$ for $1 \leq i < \mbox{min}\{m,n\}$.
\end{enumerate}
The matrix $S$ is unique up to the sign of the $S(i,i)$ and is called the \textit{Smith normal form} of $M$.  The integers $S(i,i)$ are known as the \textit{invariant factors} of the matrix $M$, for reasons we now explain.  Viewing the matrix $M$ as defining a homomorphism of free abelian groups
\[
M \colon \ZZZ^{n} \to \ZZZ^{m},
\]
the cokernel $\ZZZ^{m} / \im(M)$ of this map has as its invariant factor decomposition \cite[Chap. 12, Theorem 5]{D-F} %made reference more specific
\[
\prod \ZZZ / S(i,i)\ZZZ.
\]
We can further decompose this cokernel into cyclic groups of prime power order--its elementary divisor decomposition--and for this reason the prime power factors of the invariant factors of $M$ are known as the \textit{elementary divisors} of $M$.  The purpose of going through this terminology is to stress that when one tries to find the Smith normal form of an integer matrix, one is really seeking a description of this cokernel, and this is a problem that can be solved one prime at a time.  

Returning now to graphs, we remark that when we are looking at the adjacency matrix this cokernel goes by the name of the \textit{Smith group} of the graph \cite{rushanan-combo}.  The torsion subgroup of the Laplacian cokernel has many names in the literature \cite{lorenzini-snf}, one of which is the \textit{critical group} of the graph.  The critical group of a graph is especially interesting because it has geometric and combinatorial interpretations:  the order of the critical group is the number of spanning forests in the graph, a connected graph is a tree if and only if its Laplacian is `unimodularly congruent' to its Smith normal form \cite{L-P-W-X}, etc.
\section{Eigenvalues as Character Sums}
We return to the situation described in the introduction.  Thus $G$ is a finite abelian group, $E$ is a subset of $G$, and $\CC$ is the corresponding Cayley graph.  Denote by $A_{E}$ the adjacency matrix of this graph.

In what follows we will require some basic familiarity with characters.  Most of what we need can be found in, for example,  \cite[Chaps. 2 and 3]{isaacs}. %added reference
 The following is a well-known result expressing the eigenvalues of $A_{E}$ in terms of the irreducible complex characters of $G$, and was first observed in \cite{M-M}.  For completeness we provide a short proof.

\begin{theorem} \label{thm-main}
Let $\irred(G)$ denote the set of irreducible complex characters of $G$.  Let $M$ denote the character table of $G$, with rows indexed by $\irred(G)$ and columns indexed by $G$ in the same order as for $A_{E}$.  %reminded that rows of M are indexed by Irred(G)
Then
\begin{equation} \label{A}
\frac{1}{|G|} M A_{E}^{t} \overline{M}^{t} = \diag\left ( \sum_{e \in E} \chi(e)\right )_{\chi \in \irred(G)}.
\end{equation}
Thus the eigenvalues of $A_{E}$ take the form $\sum_{e \in E} \chi(e)$, for $\chi \in \irred(G)$.
\end{theorem}

\begin{proof}
Observe that 
\begin{eqnarray*}
MA_{E}^{t}(\chi,g) &=& \sum_{\substack{h \in G\\ hg^{-1} \in E}} \chi(h) \\
&=& \sum_{e \in E} \chi(eg).
\end{eqnarray*}
Thus we have
\begin{eqnarray*}
MA_{E}^{t}\overline{M}^{t}(\chi, \psi) &=& \sum_{g \in G} \sum_{e \in E} \chi(eg) \overline{\psi(g)} \\
&=& \sum_{e \in E} \chi(e) \cdot \sum_{g \in G} \chi(g) \overline{\psi(g)} \\
&=& \begin{cases}
  |G| \sum_{e \in E} \chi(e), & \mbox{if $\chi =\psi$}\\
  0, &  \mbox{otherwise}
  \end{cases} 
\end{eqnarray*}
where the last equality follows from the orthogonality relations.  This proves equation (\ref{A}).  Again by the orthogonality relations we have that $\frac{1}{|G|}M \overline{M}^{t} = I$; from this the final statement follows.
\end{proof}

Thus finding the spectrum of $A_{E}$ is reduced to computing the character sums $\sum_{e \in E}\chi(e)$.  Generally speaking, the spectrum of an integer matrix has little %changed ``no connection'' to ``little connection''
 connection to its elementary divisors.  See \cite{rushanan-ev, sin} for a discussion in the context of the adjacency matrix, and \cite{lorenzini-snf} for what can be said about the Laplacian.  For abelian Cayley graphs the connection to the Smith normal form of $A_{E}$ is made by an important observation of Sin \cite[p. 1364]{sin}, which we paraphrase in the following two theorems.  

Observe that since $A_{E}$ is a zero-one matrix, we can view its entries as coming from any commutative ring. In what follows $K$ will denote an algebraic closure of the field of $p$-adic numbers, $\Qp$.  We let $\zeta \in K$ denote a primitive $|G|$-th root of unity.  The ring of $p$-adic integers is denoted $\Zp$ and we set $R = \Zp[\zeta]$.  Recall that a prime element $\pi \in R$ is said to \textit{lie over} the prime $p \in \Zp$ if $\pi R \cap \Zp = p\Zp$. %several changes here.  reworded the end of the previous paragraph, and added this paragrph.  I split the next theorem up into two theorems, and was more careful with their proofs.

\begin{theorem} \label{thm_sin_gen}
Let $p$ be a prime integer that does not divide $|G|$, let $i$ be a positive integer.  Let $\pi \in R$ be a prime lying over $p \in \Zp$.  Then the multiplicity of $p^{i}$ as an elementary divisor of $A_{E}$ is equal to the number of eigenvalues of $A_{E}$ exactly divisible by $\pi^{i}$ in $R$.
\end{theorem}

\begin{proof}
Our preliminary facts and terminology about the Smith normal form carry over with very slight modification when one replaces the integers with any principal ideal domain.  It is clear that the multiplicity of $p^{i}$ as an elementary divisor of $A_{E}$ remains the same when one views the matrix entries as coming from the ring of $p$-adic integers $\Zp$.  Now let $\zeta$ be a primitive $|G|$-th root of unity in an algebraic closure $K$ of the field of $p$-adic numbers, $\Qp$, and consider the ring $R = \Zp[\zeta]$.  The prime $p$ is unramified in the extension $\Zp \subset R$ since $p \nmid |G|$, hence the multiplicity of $p^{i}$ as an elementary divisor  of $A_{E}$ over $\Zp$ is the same as the multiplicity of $\pi^{i}$ as an elementary divisor of $A_{E}$ over $R$ for any prime $\pi$ of $R$ lying over $p$.  Since $\overline{M}(\chi,g) = M(\chi,g^{-1})$, the matrices in Theorem \ref{thm-main} can be viewed as having entries from $R$, and since $\frac{1}{|G|}M \overline{M}^{t} = I$ we see that equation (\ref{A}) defines an equivalence of matrices over $R$.  The theorem follows.%cleaned up proof
\end{proof}

\begin{remarks}
\begin{enumerate}
\hfil
\item It is obvious that the conclusion of Theorem \ref{thm_sin_gen} remains true if we replace $A_{E}$ by any matrix diagonalized by $M$.  This includes linear combinations of the $A_{E}$ and the identity matrix; in particular, the Laplacian, signless Laplacian, Seidel adjacency matrix, etc.
\item See an example in Section \ref{non-integral} below of when $A_{E}$ has non-integer eigenvalues.  However, the most common situation we will encounter is when all of the eigenvalues of $A_{E}$ are integers.  In this case we have the following useful result.%reworded one of the remarks, and changed their order.
\end{enumerate}
\end{remarks}

\begin{theorem} \label{thm_sin}
Let $p$ be a prime integer that does not divide $|G|$, let $i$ be a positive integer.  Suppose that the eigenvalues of $A_{E}$ are all integers.  Then the multiplicity of $p^i$ as an elementary divisor of $A_{E}$ is the same as the number of eigenvalues of $A_{E}$ exactly divisible by $p^{i}$.
\end{theorem}

\begin{proof}
An integer not divisible by $p$ becomes a unit in $\Zp$, hence will also not be divisible by any $\pi \in R$ lying over $p \in \Zp$.  The result now follows from Theorem \ref{thm_sin_gen}.
\end{proof}

We now apply these theorems in conjunction to obtain strong results about elementary divisors for a variety of examples.
\section{Applications}
Since $G$ is a finite abelian group, it is isomorphic to a direct product of cyclic groups.  We will form interesting Cayley graphs by using a fixed cyclic decomposition of $G$ to define our connecting set.  One construction will be used so often that we define it now.  Under the identification
\[
G = Z_{q_{1}} \times Z_{q_{2}} \times \cdots \times Z_{q_{n}},
\]
where $Z_{q}$ denotes the (multiplicative) cyclic group of order $q$, define the connecting sets $E_{k}$, for $0 \le k \le n$:
\[
E_{k} \coloneqq \{ g \in G \, | \, \mbox{exactly } k \mbox{ components of } g \mbox{ are not equal to the identity}\}.
\]
When we are dealing with a Cayley graph defined by a connecting set $E_{k}$, we will write $A_{k}$ instead of $A_{E_{k}}$ for the adjacency matrix.  We denote by $[n]$ the set $\{1, 2, \ldots , n\}$ and we denote by $\binom{[n]}{k}$ the collection of subsets of $[n]$ of size $k$. %added notation explanation

We can provide a reasonable description of the eigenvalues $\sum_{e \in E_{k}} \chi(e)$ of $A_{k}$.  Let $\chi \in \irred(G)$.  Then, for $g = (g_{1}, g_{2}, \ldots , g_{n}) \in G$, %changed \cdots to \ldots
$\chi$ takes the form
\[
\chi(g) = \chi_{1}(g_{1}) \chi_{2}(g_{2}) \cdots \chi_{n}(g_{n})
\]
for some $\chi_{i} \in \irred(Z_{q_{i}})$. 

We have that
\begin{equation} \label{B}
\begin{aligned}
\sum_{e \in E_{k}} \chi(e) &= \sum_{(e_{1}, e_{2}, \ldots , e_{n}) \in E_{k}} \chi_{1}(e_{1}) \chi_{2}(e_{2}) \cdots \chi_{n}(e_{n}) \\ %changed \cdots to \ldots underneath second sum
&= \sum_{K \in \binom{[n]}{k}} \prod_{i \in K} \sum_{\substack{e_{i} \in Z_{q_{i}} \\ e_{i} \neq 1}} \chi_{i}(e_{i})
\end{aligned}
\end{equation}
and by considering the inner product of $\chi_{i}$ with the principal character of $Z_{q_{i}}$ we see that
\[
\sum_{\substack{e_{i} \in Z_{q_{i}} \\ e_{i} \neq 1}} \chi_{i}(e_{i}) = \begin{cases}
  q_{i} - 1,  & \mbox{if $\chi_{i}$ is principal}\\
  -1, &  \mbox{otherwise.}
  \end{cases} 
  \]
  \subsection{The Hamming association scheme.}
 
 Let $H(n,q)$ denote the Hamming association scheme; that is, $H(n,q)$ consists of $n$-tuples with coordinates coming from an alphabet of size $q$.  Two such tuples are $k$-th associates if they differ in exactly $k$ coordinate positions.  Setting $G = Z_{q} \times Z_{q} \times \cdots \times Z_{q}$ ($n$ times), we see that the distance $k$ association matrix of $H(n,q)$ is precisely $A_{k}$.  We remark again that our approach applies not just to the adjacency matrices but to any matrix in the Bose-Mesner algebra \cite[p. 9]{delsarte} of $H(n,q)$. %added reference

 In this case the value of (\ref{B}) depends only on the number of $\chi_{i}$ that are principal.  Explicitly, if exactly $\ell$ of the $\chi_{i}$ are principal, (\ref{B}) collapses to express the eigenvalues in their usual form as integer values of the Krawtchouk polynomials \cite[p. 38]{delsarte}: %made reference more specific
 \begin{equation} \label{C}
 \sum_{(e_{1}, e_{2}, \cdots , e_{n}) \in E_{k}} \chi_{1}(e_{1}) \chi_{2}(e_{2}) \cdots \chi_{n}(e_{n}) = \sum_{j=0}^{k} \binom{\ell}{j} \binom{n - \ell}{k - j} (q-1)^{j} (-1)^{k - j}.
 \end{equation}
The number of $\chi \in \irred(G)$ consisting of exactly $\ell$ copies of the principal character of $Z_{q}$ is $\binom{n}{\ell} (q-1)^{n - \ell}$, and we can apply Theorem \ref{thm_sin} to (\ref{C}) to compute the $p$-elementary divisor multiplicities of $A_{k}$, for primes $p$ not dividing $q$.

It is often the case that many of the terms in (\ref{C}) are equal to zero.  In particular, consider the maximal distance association matrix $A_{n}$.  From (\ref{C}) we see that the eigenvalues of $A_{n}$ are $(q-1)^{\ell}(-1)^{n-\ell}$ occurring with multiplicity $\binom{n}{\ell} (q-1)^{n - \ell}$.  Since a prime that divides $q$ will not divide $q-1$, we see that all of the elementary divisors of $A_{n}$ can be obtained from the spectrum in this case; and, in fact, the invariant factors of $A_{n}$ are \textit{equal} to its eigenvalues.  This fact was conjectured in \cite{JMU} and first proved in \cite{sin}.

If we restrict ourselves to the binary Hamming scheme $H(n, 2)$, then the situation becomes simpler.  The association matrices $A_{k}$ and $A_{n-k}$ are the same up to row permutation, hence they share the same Smith normal form.  The distance $1$ matrix is then the adjacency matrix for the well-studied $n$-cube graph.

\subsection{The $n$-cube graph.} 

Here $G = Z_{2} \times Z_{2} \times \cdots \times Z_{2}$ ($n$ times), $E = E_{1}$, and $A = A_{1}$ is the adjacency matrix of the $n$-cube graph.  Thus Theorem \ref{thm_sin} applies to give us the $p$-elementary divisors of $A$ for odd primes $p$.  Here (\ref{C}) collapses to give us eigenvalues 
\begin{equation} \label{ncube_spec}
-n +2\ell
\end{equation}
occurring with multiplicity $\binom{n}{\ell}$, for $0 \leq \ell \leq n$.  These eigenvalues have come up in many applications \cite[Chap. 7]{vanlint}, \cite{stanley}. %made reference more specific
 We see that when $n$ is odd, all of the eigenvalues of $A$ are odd (and $A$ is nonsingular).  Thus $A$ has only odd elementary divisors and so the complete structure of the Smith group of the $n$-cube can be deduced from the eigenvalues in this case.  When $n$ is even, however, the $2$-primary component of the Smith group will be nontrivial.  We will return to the $2$-part of $A$ in our discussion of Laplacians below.
 
 \subsection{Cartesian products of complete graphs.}
 
 Generalizing the $n$-cube graph we now set $G = Z_{q_{1}} \times Z_{q_{2}} \times \cdots \times Z_{q_{n}}$ but continue to use connecting set $E = E_{1}$.  We again write $A = A_{1}$ for the adjacency matrix. Cayley graphs of this form are precisely the Cartesian products of complete graphs.
 
 To describe the eigenvalues corresponding to $\chi = (\chi_{1}, \chi_{2}, \cdots , \chi_{n}) \in \irred(G)$, we need to know not just how many of the $\chi_{i}$ are principal but their locations as well.  Say $\chi_{i_{1}}, \chi_{i_{2}}, \cdots, \chi_{i_{\ell}}$ are the principal ones.  Then (\ref{B}) simplifies to
 \begin{equation} \label{cart_spec}
 -n + \sum_{j = 1}^{\ell} q_{i_{j}}
 \end{equation}
 and there are $\prod_{i \notin \{i_{1}, i_{2}, \cdots , i_{\ell} \}} (q_{i}-1)$ characters $\chi \in \irred(G)$ of this form.  From this we can deduce the $p$-elementary divisors of $A$ for primes $p$ that divide none of the $q_{i}$, $0 \leq i \leq n$.
 
 \subsection{The Laplacian and the critical group.}

As the Cayley graph is regular with valency $|E|$, the adjacency spectrum determines the Laplacian spectrum. We now recover some powerful results on the critical group of some of the Cayley graphs above. 

\subsubsection{The critical group of the $n$-cube.}

Let $L$ denote the Laplacian matrix of the $n$-cube.  Then equation (\ref{ncube_spec}) implies that the eigenvalues of $L$ are $2\ell$ occurring with multiplicity $\binom{n}{\ell}$, for $0 \leq \ell \leq n$.  It follows that, for odd primes $p$, the Sylow  $p$-subgroup of the critical group of the $n$-cube is isomorphic to
\begin{equation} \label{bai_result}
\prod_{\ell=1}^{n} Syl_{p}(Z_{\ell})^{\binom{n}{\ell}}.
\end{equation}
This is the main result of \cite{bai}.  

In general, the full structure of the $2$-primary component of both the critical group and the Smith group of the $n$-cube remain unknown.  In \cite[Theorem 1.3]{bai},%made reference more precise
 the $2$-rank of $L$ is shown to be equal to $2^{n-1}$, and the multiplicity of $2$ as an elementary divisor of $L$ is also determined.  As we mentioned earlier, for odd $n$ the $2$-rank of the adjacency matrix $A$ is $2^{n}$.  For even $n$, note that by definition of the Laplacian we have that $L \equiv A \pmod 2$.  Hence, for even $n$, the $2$-rank of $A$ is also $2^{n-1}$.  More generally, if $L \equiv A \pmod {2^{i}}$ then the multiplicity of $2^{j}$, $0 \leq j < i$, as an elementary divisor is the same for both $A$ and $L$ \cite[Lemma 3.3]{B-D-S}.  Computer calculations seem to indicate that the $2$-primary component of the Smith group may be easier to understand.  We conjecture that the multiplicity of $2^{i}$ as an elementary divisor of $A$ is equal to the number of eigenvalues of $A$ exactly divisible by $2^{i+1}$.

\subsubsection{The critical group of a Cartesian product of complete graphs.}

The previous result on the $p$-primary component of the critical group of the $n$-cube, for odd primes $p$, was generalized to the critical group of a Cartesian product of complete graphs \cite[Theorem 1.2]{J-N-R}.%made reference more precise
  Recall that such a graph is the Cayley graph $\CC$ of $G = Z_{q_{1}} \times Z_{q_{2}} \times \cdots \times Z_{q_{n}}$ with connecting set $E = E_{1}$.  Let $A = A_{1}$ denote the adjacency matrix and let $L$ denote the Laplacian.  From equation (\ref{cart_spec}) we deduce that each subset $\{i_{1}, i_{2}, \cdots , i_{\ell}\}$ of $[n]$ determines $\prod_{i \notin \{i_{1}, i_{2}, \cdots , i_{\ell} \}} (q_{i}-1)$ eigenvalues of $L$ of the form
\[
n - \sum_{j=1}^{\ell} q_{i_{j}} + \sum_{i = 1}^{n}(q_{i} - 1) = \sum_{i \notin \{i_{1}, i_{2}, \cdots , i_{\ell}\}} q_{i}.
\]
It follows that, for a prime $p$ not dividing any of the $q_{i}$, $1 \leq i \leq n$, the Sylow $p$-subgroup of the critical group of $\CC$ is isomorphic to
\[
\prod_{\substack{S \subseteq [n]\\ S \neq [n]}} Syl_{p}(Z_{\sum_{i \notin S} q_{i}})^{\prod_{i \notin S}(q_{i}-1)}.
\]
This result was proved in \cite{J-N-R}.  It is worth mentioning that the proof in \cite{J-N-R} relies heavily on keeping track of integral row and column operations on $L$.  Our proof of (\ref{bai_result}) is also of a very different nature than the proof in \cite{bai}.

\subsection{Non-integer eigenvalues.} \label{non-integral}%changed "integral" to "integer"

It was observed in \cite[Example 4.1]{K-S} %made reference more precise
that if we take our connecting set $E$ to be any union of the $E_{i}$ then the eigenvalues of $A_{E}$ are all integers.  We conclude with a small non-integer example.%changed "integral" to "integer"

Let $G = Z_{7} = \langle  x \rangle$ and use connecting set $E = \{x^{4}, x^{5}, x^{6}\}$.  Following the notation preceding Theorem \ref{thm_sin_gen}, %referenced previous notation
$\zeta \in K$ is a primitive $7$-th root of unity and set $\alpha = \zeta^{5} + \zeta^{2} + \zeta + 1$ and $\beta = \zeta^{5} + \zeta^{4} + \zeta^{3} + 1$.  Notice that in the ring $R = \mathbb{Z}_{2}[\zeta]$ we have $2 = \zeta^{2} \cdot \alpha \cdot \beta$.  The adjacency matrix $A_{E}$ has seven distinct eigenvalues:  
\[
3, \quad -\alpha, \quad -\zeta^{2}\alpha, \quad -\zeta^{3} \beta, \quad -\zeta^{6}\alpha, \quad -\beta, \quad -\zeta^{2}\beta.
\]
From Theorem \ref{thm_sin_gen} %removed parenthetical statement, quoted correct theorem
we deduce that $3$ is an elementary divisor of $A_{E}$ with multiplicity $1$ and $2$ is an elementary divisor of $A_{E}$ with multiplicity $3$.

\section{Acknowledgements}
This work was supported by the James Madison University Program of Grants for Faculty Assistance.  The authors also acknowledge helpful comments from Peter Sin, Dino Lorenzini, and the anonymous referee. %updated acknowledgements to include a more specific statement of grant support, and to thank the referee

\bibliographystyle{plain}
\bibliography{Bibliography}

\end{document}